\documentclass[12
pt,twoside]{amsart}
\usepackage{amssymb}
\usepackage{a4wide}
\usepackage[latin1]{inputenc}

\usepackage[T1]{fontenc}
\usepackage{times}

\def\hpq0{h^{p,q}_{\leq 0}}
\def\Hpq0{\H_{\leq 0}^{p,q}}

\def\dbar{\bar\partial}
\def\ddbar{\partial\dbar}

\def\R{{\mathbb R}}

\def\C{{\mathbb C}}

\def\H{{\mathcal H}}
\def\E{{\mathcal E}}

\def\Re{{\rm Re\,  }}
\def\Im{{\rm Im\,  }}

\def\p{\dot{\phi^{t}}}
\def\be{\begin{equation}}
\def\ee{\end{equation}}
\def\om{\omega^{t}_n}

\newtheorem{thm}{Theorem}[section]
\newtheorem{lma}[thm]{Lemma}

\newtheorem{prop}[thm]{Proposition}

\theoremstyle{definition}

\theoremstyle{remark}

\newtheorem{preremark}{Remark}
\newtheorem{preex}{Example}

\numberwithin{equation}{section}

\begin{document}

\title[]
{ Probability measures associated to geodesics in the space of
  K\"ahler metrics.}

\author[]{ Bo Berndtsson}

\address{B Berndtsson :Department of Mathematics\\Chalmers University
  of Technology 
  and the University of G\"oteborg\\S-412 96 G\"OTEBORG\\SWEDEN,\\}

\email{ bob@math.chalmers.se}
\begin{abstract}
{We associate certain probability measures on $\R$ to geodesics in the
space $\H_L$ of positively curved metrics on a line bundle $L$, and to
geodesics in the finite 
dimensional symmetric space of hermitian norms on $H^0(X, kL)$. We
prove that the measures associated to the finite dimensional spaces
converge weakly to the measures related to geodesics in $\H_L$ as $k$
goes to infinity. The
convergence of second order moments  implies a recent result of Chen and
Sun on geodesic distances in the respective spaces, while the
convergence of first order moments gives convergence of Donaldson's
$Z$-functional to the Aubin-Yau energy. We also include a
result on approximation of infinite dimensional geodesics by Bergman
kernels which generalizes work of Phong and Sturm.}
\end{abstract}

\bigskip

\maketitle

\section{Introduction}

Let $X$ be a compact K\"ahler manifold and $L$ an ample line bundle over
$X$. If $\phi$ is a hermitian metric on $L$ with positive curvature,
then 

$$
\omega^\phi:= i\ddbar \phi
$$
is a K\"ahler metric on $X$ with K\"ahler form in the Chern class of
$L$, $c(L)$ , and
we let $\H_L$ denote the space of all such  K\"ahler potentials. By the
work of Mabuchi, Semmes and Donaldson (see \cite{Mabuchi},
\cite{Semmes}, \cite{Donaldson1}), $\H_L$ can be given the structure
of an infinite dimensional, negatively curved Riemannian manifold, or
even symmetric space. With this space one can associate certain finite
dimensional symmetric spaces in the following way. Take a positive
integer $k$ and let $V_k$ be the vector space of global holomorphic
sections of $kL$,
$$
V_k= H^0(X, kL).
$$
(Later we shall consider more generally vector spaces $H^0(X, kL + F)$
where $F$ is a fixed line bundle, but for simplicity we omit $F$ in
this introduction.) The finite dimensional symmetric spaces in question
are then the spaces $\H_k$ of hermitian norms on $V_k$.

There are for any $k$  natural maps
$$
FS = FS_k : \H_k \mapsto \H_L,
$$
and 
$$
Hilb=Hilb_k : \H_L \mapsto \H_k,
$$
and a basic ide in the study of K\"ahler metrics on $X$ with K\"ahler form
in $c(L)$ is that under these maps the finite dimensional spaces
$\H_k$ should approximate $\H_L$ as $k$ goes to infinity. This will be
explained a bit more closely in the next section of this paper, see also
\cite{ Donaldson1}, \cite{Phong-Sturm} and \cite{ Chen-Sun} for 
excellent backgrounds 
to these ideas. 

The most basic result in this direction is the result of Bouche,
\cite{Bouche} and Tian, \cite{Tian} that for $\phi$ in $\H_L$ 
$$
\phi_k:= FS_k\circ Hilb_k(\phi)
$$
tends to $\phi$  together with its derivatives. 
It is natural to ask whether geodesics between points in $\H_L$ also
can be approximated in some sense by geodesics coming from the finite
dimensional picture. This question was first adressed by Phong and
Sturm in \cite{Phong-Sturm}, where it is proved that any geodesic in
$\H_L$ is a limit of $FS_k$ of
geodesics in $\H_k$, in an almost uniform way (see below). Later, this
result has 
been refined in particular cases (like toric varieties) to give
convergence of derivatives as well by Song-Zelditch,
Rubinstein-Zelditch and Rubinstein, see \cite{Song-Zelditch},
\cite{Rubinstein-Zelditch}, \cite{Rubinstein}. (These works also treat
more general equations than  the geodesic equation.)

In a recent very interesting paper, \cite{Chen-Sun} , Chen and Sun
have shown that moreover if $\phi^0$ and $\phi^1$ are two K\"ahler
potentials in $\H_L$, then the geodesic distance, suitably normalized,
between $Hilb_k(\phi^0)$ 
and $Hilb_k(\phi^1)$ in $\H_k$ tends to the geodesic distance between
$\phi^0$ and $\phi^1$ in $\H_L$. Hence $\H_k$ approximates $\H_L$ as
metric spaces in this sense. 

In this paper we associate to  geodesics, in $\H_k$ and $\H_L$
respectively, certain probability measures on $\R$ from which many
quantities related to the geodesic (like length, energy) can be
recovered. The main result of the paper is that the measures
associated to geodesics in $\H_k$ converge to their counterparts in
$\H_L$  in the weak *-topology as $k$ goes to infinity. It follows that
their moments converge, which applied to second moments  implies the
result of Chen and Sun on convergence of geodesic distance. 

Let $H^0_k$ and $H^1_k$ be two points in $\H_k$, and let 
$H^t_k $ be the geodesic in $\H_k$ connecting them. The tangent vector
to this geodesic
$$
A_{t, k}:= (H^t_k)^{-1} \dot{H^t_k}
$$
is then an endomorphism of $V_k$. The geodesic condition means that
it is actually independent of $t$ so we will omit the $t$ in the
subscript. Since $A_k$ is hermitian for the
scalar products in the curve all its eigenvalues are real. Let
$\nu_k=\nu_{A_k}$ be the  normalized spectral measure of $ k^{-1}
A_k$. By this we mean that
$$
\nu_k= d_k^{-1}\sum \delta_{\lambda_j},
$$
where $\lambda_j$ are the eigenvalues of $k^{-1} A_k$ and $d_k$ is the
dimension of $V_k$, so that $\nu_k$ are probability measures on $\R$. 

The second order moment of $\nu_k$ is precisely the norm
squared 
of the vector $A_k$ in the tangent space of $\H_k$, divided by
$d_k$. Since this is 
independent of $t$ and $t$ goes from 0 to 1, the second order moment
equals the square of the normalized geodesic distance between $H^0_k$ and
$H^1_k$. We shall 
also see in section 2 that the first order moment of $\nu_k$ equals
the Donaldson functional
$$
Z(H^0_k, H^1_k)/d_k
$$
from \cite{Donaldson2}.

We next describe the corresponding objects for the infinite
dimensional space $\H_L$. Let $\phi^0$ and $\phi^1$ be two points in
$\H_L$ and let $\phi^t$ be the Monge-Ampere geodesic joining them. By
this we mean that $\phi^t$ is a curve of positively curved metrics on
$L$  for $t$
between 0 and 1. We extend the definition of $\phi^t$ to complex $t$ in
$$
\Omega:=\{0<\Re t<1\}
$$
by letting it be indepent of the imaginary part of $t$. The geodesic
equation is then 
$$
(i\ddbar \phi^t)^{n+1}=0
$$
on $\Omega\times X$.

 It
was proved by Chen in \cite{Chen} that such a geodesic always exists
and is of class $C^{ 1,1}$ in the sense that all $(1,1)$-derivatives 
are uniformly bounded. It is unknown if the geodesic is actually
smooth. A 'geodesic in $\H_L$' is therefore not necessarily a curve in
$\H_L$ (which consists of smooth metrics), but we will adhere to
the common terminology nevertheless.  For each $t$ fixed we can now define a
probability measure on $\R$ in the following way. Let first $dV_t$ be
the normalized volume measure on $X$ induced by $\omega^{\phi^t}$,
$$
dV_t:= \omega^{\phi^t}_n/Vol.
$$
Here $\omega_n:=\omega^n/n!$ for $(1,1)$-forms $\omega$ and Vol is the
volume of $X$ 
$$
Vol=\int_X c(L)_n.
$$
Since $\dot{\phi^{t}} $ is a continuous real valued function, we can
consider the direct image (or 'pushforward') of $dV_t$
\be
\mu_t = (-\p)_*(dV_t)
\ee
so that $\mu_t$ is a probability measure on $\R$. 
Concretely, this means that if $f$ is a continuous function on $\R$,
then 
$$
\int_{\R} f(x) d\mu_t(x) = \int_X f(-\p) dV_t.
$$
We shall show in
the next section that if $\phi^t$ is a Monge-Ampere geodesic, then
$\mu=\mu_t$ is  independent of $t$. This is then the measure
that corresponds to the spectral measures $\nu_k$ in the infinite
dimensional setting, and our main results says that $\nu_k$ converge
to $\mu$ in the weak* topology as $k$ goes to infinity. 

\begin{thm} Let $\phi^0$ and $\phi^1$ be two points in $\H_L$ and let 
$$
H^t_k= Hilb_k(\phi^t)
$$
for $t=0,1$ be the corresponding norms in  $\H_k$. Let for $t$ between
0 and 1 $H^t_k$ be the geodesic in $\H_k$ connecting these two norms
and let $\nu_k$ be their normalized spectral measures as defined
above. Then 
$$
\nu_k \longrightarrow \mu,
$$
 in the weak* topology, where
$\mu=\mu_t$ is defined in 1.1. 
\end{thm}

Just like the spectral measures of the endomorphisms $A_k$ contain
part of the properties of the corresponding geodesics in $\H_k$, part
of the properties of the Monge-Ampere geodesic can be read off from  the
measure $\mu$. It is for instance immediately clear that the second
order moment of $\mu$ is equal to
$$
\int_X \p^2 dV_t/Vol
$$
which is the length square of the tangent vector to the Monge-Ampere geodesic
( which is independent of $t$ as it should be). Since the parameter
interval is from 0 to 1 the length of the tangent vector is the length
of the geodesic from $\phi^0$ to $\phi^1$. By a theorem of Chen,
\cite{Chen}, the length of the geodesic is equal to the geodesic
distance, so the convergence of second order moments implies the
theorem of Chen and Sun, \cite{Chen-Sun} that normalized geodesic distance in
$\H_k$ converges to geodesic distance in $\H_L$. Similarily we shall
see in the next section that the first order moment of $\mu$ is the
Aubin-Yau energy of the pair $\phi^0$ and $\phi^1$, and convergence of
first order moments therefore says that the Aubin-Yau energy is the
limit of Donaldson's $Z$-functional (this is a much simpler result). 

The proof of our main result is given in section 3; it is based on the
curvature estimates from \cite{Berndtsson}. The basic idea is as
follows: The Monge-Ampere geodesic $\phi^t$ induces a certain curve of
norms in $\H_k$, $H_{\phi^t, k}$. These are $L^2$-norms on
  the space of global sections, similar to the curves
  $Hilb_k(\phi^t)$ but defined slightly differently to fit with the results of
  \cite{Berndtsson} . At the end points, $t=0,1$,
$$
 H_{\phi^t, k}= H^t_k:= Hilb_k(\phi),
$$
and we define $H^t_k$ for $t$ between 0 and 1 to be the geodesic in
$\H_k$ between these endpoint values. The main result of
\cite{Berndtsson} immediately implies that
$$
 H_{\phi^t, k} \geq H^t_k
$$
for $t$ between 0 and 1, and by definition equality holds at the
endpoints.
Let
$$
T_{t,k}:=H_{\phi^t, k}^{-1}\dot{H}_{\phi^t, k}
$$
Differentiating with respect to $t$ at $t=0, 1$ we then get that
$$
\langle A_k u,u\rangle_{H^0_k}\leq \langle T_{0, k} u,u\rangle_{H^0_k}
$$ 
and 
$$
\langle A_k u,u\rangle_{H^1_k}\geq \langle T_{1, k} u,u\rangle_{H^1_k}
$$ 
This means that we get estimates for the tangent vector to the finite
dimensional geodesic in terms of certain operators on $V_k$ defined by
the Monge-Ampere geodesic. These operators are Toepliz operators on
$V_k$ with symbol $\p$, $t=0, 1$ and their spectral measures are 
essentially known to converge to $\mu_t=\mu$ . Since $A_k$ is pinched
between these two operators it is not hard to see that the spectral
measures of $A_k$ have the same limit, which proves the theorem.

In a final section  we will give a result on the uniform convergence
of $FS_k$ of finite dimensional geodesics to Monge-Ampere
geodesics, generalizing the work of Phong-Sturm mentioned
earlier. This result is only a small variation of Theorem 6.1 from 
\cite{Berndtsson}, but it has as a consequence the following theorem
which is more natural than Theorem 6.1 in \cite{Berndtsson} so it seems
good to state it explicitly. 

\begin{thm} Let $\phi^0$ and $\phi^1$ be two K\"ahler potentials  in
  $\H_L$ and let $\phi^t$ be the Monge-Ampere geodesic joining
  them. Let
$$
H_k^t =Hilb_k(\phi^t)
$$
for $t=0, 1$ and let $H_k^t$ for $t$ between 0 and 1 be the
geodesic in $\H_k$ between these two points. Let finally
$$
B_{t,k}:= FS_k(H^t_k)
$$
for $0\leq t\leq 1$. Then
$$
\sup| k^{-1}\log B_{t, k} -\phi^t|\leq C\frac{\log k}{k}.
$$
\end{thm}
This theorem strengthens the main result of Phong and Sturm,
\cite{Phong-Sturm},
who proved that
$$
\lim_{l\rightarrow\infty}\sup_{k\geq l} k^{-1}\log B_{t, k}=\phi^t
$$
almost everywhere. 

The final parts of this work (the most important parts!) were carried out 
during the conference on extremal K\"ahler metrics at BIRS June-July
2009. I am grateful to the organizers for a very stimulating
conference. I would also like to thank Jian Song for suggesting that
my curvature estimates might be relevant in connection with the
Chen-Sun theorem and for encouraging me to write down the details of
the proof of Theorem 1.2. Finally I am grateful to Xiuxiong Chen and
Song Sun for explaining me their result. 

\section{Background and definitions}

In the first subsection we will give basic facts about the space
$\H_L$ and its finite dimensional 'quantizations'. Since this material
is well known (see e g  \cite{Donaldson1}, \cite{Phong-Sturm} or
\cite{Chen-Sun}) we will 
be brief and emphazise a few particularities that are relevant for
this paper.

\subsection{ $\H_L$, $ \H_k$ and its variants.}

\bigskip

Let $L$ be an ample line bundle over the compact manifold $X$. $\H_L$
is the space of all smooth metrics $\phi$ on $L$ with
$$
\omega^\phi:= i\ddbar\phi >0.
$$
$\H_L$ is an open subset of an affine space and its  tangent space at
each point equals the space of smooth real valued functions on
$X$. The Riemannian norm on this tangent space at the point $\phi$ is
the $L^2$-norm
$$
\|\psi\|^2 = \int_X |\psi|^2 \omega^\phi_n/Vol
$$
(remember we use the notation $\omega_n=\omega^n/n!$ for forms of
degree two). A geodesic in $\H_L$ is a curve $\phi^t$ for $a<t<b$ that
satisfies the geodesic equation
\be
\frac{d^2}{dt^2}\phi^t =|\dbar\frac{d}{dt}\phi^t|^2_{\omega^{\phi^t}}.
\ee
It is useful to extend the definition of $\phi^t$ to complex values of
$t$ in the strip
$$
\Omega=\{t; a<\Re t<b\}
$$
by taking it to be independent of the imaginary part of $t$. Then 2.1
can be written equivalently on complex form
$$
c(\phi^t):= \phi^t_{t \bar t} -|\dbar\p|^2_{\omega^{\phi^t}}=0,
$$
where $\p=\partial\phi^t /\partial t $. On the other hand the
expression $c(\phi^t)$ is related to the Monge-Ampere operator through
the formula
$$
c(\phi^t) idt\wedge d\bar t\wedge \omega^{\phi^t}_n =
(i\ddbar\phi^t)_{n+1},
$$
where on the right hand side we take the $\ddbar$-operator on
$\Omega\times X$. Geodesics in $\H_L$ are therefore given by solutions
to the homogenuous Monge-Ampere equation that are independent of $\Im
t$. Notice that a geodesic will automatically satisfy
$$
i\ddbar\phi^t\geq 0,
$$
and we shall refer to any  curve with this property as a 'subgeodesic'
even though this 
term has no meaning in Riemannian geometry in general.

A fundamental theorem of Chen, \cite{Chen} says that if $\phi^0$ and
$\phi^1$ are two points in $\H_L$ they can be connected by a geodesic
of class $C^{ 1,1}$, i e such that
$$
(i\ddbar\phi^t)^{n+1}=0
$$
and
$$
\ddbar \phi^t
$$
has bounded coefficients.

One associates with $\H_L$ the vector spaces
$$
V_k:= H^0(X, kL)
$$
of global holomorphic sections of $kL$ for $k$ positive integer. 
A metric $\phi$ in $\H_L$ is mapped to a hermitian norm $Hilb_k(\phi)$
on $V_k$ by
$$
\|u\|^2_{Hilb_k(\phi)} := \int_X |u|^2 e^{-k\phi}\omega^\phi_n.
$$
It will also be useful for us to consider the vector spaces 
$$
H^0(X, K_X+kL).
$$
A metric $\phi$ on $L$ also induces an hermitian norm, $H_{k\phi}$  on
these spaces through 
$$
\|u\|^2_{H_{k\phi}}:=\int_X |u|^2 e^{-k\phi}.
$$
An important point is that $|u|^2 e^{-k\phi}$ is a measure on $X$ if
$u$ lies in $H^0(X, K_X+kL)$, so the integral of this expression is
naturally defined, without the introduction of any extra measure like
$\omega^\phi_n$.

In order to treat both these types of spaces simultaneously we let $F$ be an
arbitrary line bundle over $X$ and consider spaces
$$
H^0(X, K_X+kL +F).
$$
Norms on these spaces are then defined by
$$
\|u\|^2_{H_{k\phi +\psi}}:=\int_X |u|^2 e^{-k\phi -\psi},
$$
where $\psi$ is some metric on $F$. The two cases we discussed earlier
the correspond to $F=-K_X$ and 
$$
\psi =-\log \omega^\phi_n,
$$
and $F=0$ respectively. In the first case
$$
H_{k\phi +\psi}=Hilb_{k(\phi)}
$$
as defined above. 

Let now $V$ be any space of sections to some line bundle, $G$, over $X$; it
may be any of the choices discussed above, and denote by $\H_V$ the
space of hermitian norms  on $V$.  For such a hermitian norm, $H$, let
$s_j$ be an orthonormal basis for the space of sections $H^0(X, G)$,
and consider the Bergman kernel
$$
B_H=\sum |s_j|^2.
$$
The absolute values on the right hand side here are to be interpreted
with respect to some trivialization of $G$. When the trivialization
changes, $\log B_H$ transforms like a metric on $G$ since
$$
|u|^2/B_H
$$
is a well defined function if $u$ is a section of $G$. By definition
$FS(H)$ is that metric
$$
FS(H)=\log B_H.
$$
By the well known extremal characterization of Bergman kernels we have
$$
B_H(x)=\sup_{u\in H^0(X,G)} \frac{|u(x)|^2}{\|u\|^2_H}.
$$
From this we can conclude that the Bergman kernel is a decreasing
function of the metric; if we change the metric to a larger one, the
Bergman kernel becomes smaller.

Choosing a basis for $V$ we can
represent an element in $\H_V$ by a matrix that we slightly abusively
also call $H$. A curve in $\H_V$ then
gets represented by a curve of matrices $H^t$. Differentiating norms
we get
$$
\frac{d}{dt}\|u\|^2_{H^t}=\langle A_t u, u\rangle_{H^t},
$$
with
$$
A_t = (H^t)^{-1}\frac{d}{dt} H^t.
$$
 $A_t$  is an
endomorphism of $V$; the tangent vector to the curve $H^t$. Its norm
is 
$$
\|A_t\|^2 = tr A^* A.
$$
Here the * stands for the adjoint with respect to $H$, but since $A$
is selfadjoint for this scalar product, the norm of $A$ is the sum of
the squares of its eigenvalues.

Finally,  the geodesic equation is
$$
\frac{d}{dt} A_t =0.
$$

It is easy to see that any two norms in $\H_V$ can be joined by a
geodesic. Explicitly, we can find a basis $s_j$ of $V$ which is
orthonormal w r t $H^0$ and diagonalizes $H^1$ with eigenvalues
$e^{\lambda_j}$.
 The geodesic is then represented (in this basis) by the diagonal
 matrix $H^t$ with eigenvalues $e^{t\lambda_j}$. Hence, $A=A_t$ is
 diagonalized by the same basis and has eigenvalues $\lambda_j$. 

\bigskip

Just like in the case of $\H_L$ it is convenient to consider curves
$H^t$ defined also for complex values of $t$  in the strip $\Omega$,
by letting it be independent of the imaginary part of $t$. We can then
write the geodesic equation equivalently as
$$
\frac{\partial}{\partial\bar t} H^{-1}\frac{\partial}{\partial t} H.
$$
This suggests that the geodesic equation can be thought of as the
zero-curvature equation for a certain 
vector bundle. Let $E$ be the trivial bundle over $\Omega$ with fiber
$V$. A curve in $\H_V$ is then the same thing as a vector bundle
metric on $E$, independent of the imaginary part of $t$, and we see
that geodesics correspond to flat metrics on $E$. In analogy with the
case of curves in $\H_L$, we will call curves in $\H_V$ that
correspond to vector bundle metrics of semipositive curvature 
'subgeodesics' in $\H_V$. 

\bigskip

A main role in the sequel is played by Theorem 2.1 in
\cite{Berndtsson}. This theorem implies that if $\phi^t$ is a subgeodesic in
$\H_L$ (it does not need to be independent of $\Im t$), i e 
satisfies
$$
i\ddbar \phi^t \geq 0,
$$
then the induced curve $H_{\phi^t}$ in $\H_V$ for $V=H^0(X,K_X+L)$ has
semipositive curvature, so it is a subgeodesic in $\H_V$. Since
metrics with semipositive curvature lie above flat metrics having the
same boundary values, this 
gives us a way of comparing 
$L^2$-norms on $V$ induced by (sub)geodesics in $\H_L$ to finite dimensional
geodesics in $\H_V$ (cf Proposition 3.1).  

\subsection{ Measures defined by geodesics.}

Let us start with the case of a finite dimensional geodesic, $H^t$,  in
$\H_V$. As we have seen in the previous subsection it can be
represented by a diagonal matrix with diagonal elements
$e^{t\lambda_j}$ in a suitable basis, and its tangent vector $A$ is
then diagonal with diagonal elements $\lambda_j$. The measure we
associate to the geodesic is then the (normalized) spectral measure of
$A$
$$
\nu_A=\frac{1}{d}\sum \delta_{\lambda_j},
$$
with $d$ the dimension of $V$. This is defined in terms of
eigenvalues of the endomorphism $A$ so it does not depend on the basis we
have chosen. 

\bigskip

\noindent Recall that for any pair of norms in $\H_V$, Donaldson
\cite{Donaldson2} has defined a quantity
$$
Z(H^1, H^0)= \log\frac{\det H^1}{\det H^0}
$$
(the determinant is the determinant of a matrix representing the norm
in some basis, but since we consider quotients of determinants, $Z$
does not depend on which basis). Then
$$
\frac{d}{dt} Z(H^t,H^0)= tr A.
$$
Hence we see that, since $A$ is constant and we have chosen our
parameter interval  to be $[0,1]$, that
$$
\int_\R x d\nu_A = tr A/d =Z(H^1, H^0)/d
$$
so first moments of the spectral measure gives the Donaldson
$Z$-functional. Second order moments are
$$
\int_\R x^2 d\nu_A = tr A^2/d =\|A\|^2/d
$$
which in the same way equals the square of the geodesic distance from
$H^0$ to $H^1$, again divided by $d$.

We next turn to the corresponding construction for $\H_L$. Let
$\phi^t$ be a curve in $\H_L$ and to fix ideas we think of $t$ as real
now. We first assume that $\phi^t$ is smooth and denote by
$$
\p=\frac{d\phi^t}{dt}
$$
the tangent vector (a smooth function on $X$). For ease of notation we
also set
$$
\omega^t=\omega^{\phi^t}.
$$
\begin{lma} Let $f$ be a compactly supported function on $\R$ of class
  $C^1$. Then
$$
\frac{d}{dt}\int_X f(\p)\om =\int_X f'(\p)c(\phi^t)\om.
$$

\end{lma}

\begin{proof} This is just a simple computation.
$$
\frac{d}{dt}\int_X f(\p)\om = \int f'(\p)\frac{d^2\phi^t}{dt^2}\om
+\int_X f(\p)i\ddbar\p\wedge \omega^t_{n-1}.
$$
By Stokes' theorem applied to the last term this equals
$$
 \int f'(\p)\frac{d^2\phi^t}{dt^2}\om
-\int_X f'(\p)i\partial\p\wedge\bar\partial\p\wedge \omega^t_{n-1}
=\int_X f'(\p)c(\phi^t)\om.
$$
\end{proof}

Since for smooth geodesics $c(\phi^t)=0$ it follows that the integrals
$$
\int_X f(\p)\om
$$
do not depend on $t$. By approximation we can draw the same conclusion
for (say) geodesics of class $C^1$.

\begin{prop}
Let $\phi^t$ be a curve of metrics on $L$ with semipositive curvature
which is of class $C^1$ and satisfies
$$
(i\ddbar\phi^t)^{n+1}=0
$$
in the sense of currents. Then the integrals
$$
\int_X f(\p)\om
$$
do not depend on $t$.

\end{prop}

\begin{proof} Let $K$ be a compact in $\Omega$. We can then
  approximate $\phi^t$ over $K\times X$ by smooth metrics
  $\phi^t_\epsilon$ such that
$$
i\ddbar\phi^t_\epsilon \geq 0
$$
and
$$
\int_{K\times X}(i\ddbar\phi^t_\epsilon)^{n+1}
$$
tends to 0. In fact, the approximation can be carried out locally by
convolution and then patched together with a partition of unity - the
patching causes no problem if the initial metric is of class
$C^1$. The proposition then follows from the lemma.
\end{proof}

For a $C^1$-geodesic we now consider the normalized volume measures on
$X$
$$
dV_t=\om/Vol
$$
where  
$$
Vol=\int_X c(L)_n
$$
is the volume of $X$, and their direct image measures under the map
$-\p$
$$
d\mu_t= (-\p)_*(dV_t).
$$ 
These are probability measures on $\R$, supported on a compact
interval $[-M, M]$ , $M=\sup|\p|$ and concretely defined by
$$
\int_{\R} f(x)d\mu_t(x)=\int_X f(-\p)\om/Vol.
$$
By the proposition, they do in fact not depend on $t$, so
$d\mu=d\mu_t$ is a fixed probability measure on $\R$ associated to the
given geodesic. 

\bigskip

\noindent Recall that the Aubin-Yau energy of a pair of metrics in
$\H_L$ is defined in the following way:
$$
\frac{d}{dt}\E(\phi^t, \phi^0) = -\int_X \p \omega^t_n,
$$
and $\E(\phi^0, \phi^0)=0$. From this we see that the first order
moment of $d\mu$ 
$$
\int x d\mu(x)= -\int_X \p \omega^t_n/Vol,
$$
is preciseley the derivative of the Aubin-Yau energy, which is
constant for a geodesic, and hence equal to the Aubin-Yau energy
itself if the parameter interval is $(0,1)$. This corresponds to the
relation between the measures $d\nu_k$ 
and the Donaldson $Z$-functional, and Theorem 1.1 in this case is just
the familiar convergence of the $Z$-functionals to the Aubin-Yau
energy. Similarily, the second order moments
$$
\int x^2 d\mu(x)= \int_X (\p )^2\omega^t_n/Vol,
$$
is the length of the tangent vector to $\phi^t$ squared, so second
order moments give geodesic distances. Notice finally that the
proposition implies that all $L^p$-norms of $\p$ are constant along
the curve, hence also the $L^\infty$-norm. More precisley, since $\sup
(-\p)$ is the supremum of the support of $\mu$ it follows that
$\inf\p$ (and $\sup\p$) are constant (where we mean {\it essential} sup
and inf).

\bigskip

\noindent{\bf Remark} Notice also that if we define the measures in the same
way when $\phi^t$ is a subgeodesic, then the integrals
$$
\int_{\R} f(x)d\mu_t(x)
$$
increase with $t$ if $f$ is an increasing function. Intuitively, the
measures $\mu_t$ move to the right as $t$ increases.

\section{The convergence of spectral measures}
We first state a consequence of the main result from
\cite{Berndtsson}. In the statement of the proposition we shall use
the notation
$$
\|u\|^2_{H_\phi} =\int_X |u|^2 e^{-\phi}
$$
for the hermitian norm on $H^0(X, L+K_X)$ defined by a metric $\phi$
on $L$.

\begin{prop} Let $L$ be an ample line bundle over $X$ and let $\phi^t$
  for $t=0, 1$ be two elements of $\H_L$. Let for $t=0, 1$  $H^t$ be
  the norms $H_{\phi^t}$ on  $H^0(X, L+K_X)$ defined by $\phi^0$ and
  $\phi^1$. Let for $t$ between 0 and 1 $H^t$ be the geodesic in the
  space of metrics on  $H^0(X, L+K_X)$ joining $H^0$ and $H^1$. Let
  finally $\phi^t$ be any smooth subgeodesic in $\H_L$ connecting
  $\phi^0$ and $\phi^1$, i e any  metric with nonnegative
  curvature on $L$ over $X \times \Omega$, smooth up to the boundary. Then
\be
H^t \leq H_{\phi^t}.
\ee
\end{prop} 

\begin{proof} If we regard $H^t$ and $H_{\phi^t}$ as vector bundle
  metrics on the trivial vector bundle over $\Omega$ with fiber 
  $H^0(X, L+K_X)$, then Theorem 2.1 of \cite{Berndtsson} implies that
  the second of these metrics has nonnegative curvature. On the other
  hand the first metric has zero curvature since $H^t$ is a geodesic
  . Since the two metrics agree over the boundary a comparison lemma
  from \cite{Rochberg} or \cite{Semmes} gives inequality 3.1.
\end{proof}

We have been a little bit vague about what 'smoothness' means in the
proposition. The proof of Theorem 2.1 in \cite{Berndtsson} requires at
least $C^2$-regularity, but we claim that $C^1$ regularity is
sufficient in the proposition, which can be seen from regularization
of the metric (this can be done locally with the aid of a partition of
unity in the case that the metric is $C^1$ from the start). This means
that we can (and will) apply the proposition to  Monge-Ampere
geodesics of class $C^{1,1}$.

The next step is to differentiate the inequality 3.1 for $t= 0, 1$
(recall that equality holds at the endpoints). If $u$ lies in $H^0(X, L+K_X)$ 
we get
$$
\frac{d}{dt} \|u\|^2_{H^t}=\langle A_t u,u\rangle_{H^t},
$$
where
$$
A_t = (H^t)^{-1} \dot{H^t}.
$$
Since $H^t$ is a geodesic, $A_t=A$ is independent of $t$. The
derivative of the right hand side of 3.1 is
$$
\frac{d}{dt} \|u\|^2_{H_{\phi^t}}=\langle T_t u,u\rangle_{H_{\phi^t}},
$$
where $T_t$ is the Toepliz operator on  $H^0(X, L+K_X)$ defined by
$$
\langle T_t u,u\rangle_{H_{\phi^t}}=-\int_X \p |u|^2 e^{-\phi^t}.
$$
The proposition then implies that
\be
T_0\leq A
\ee
as operators on the space $H^0(X, L+K_X)$ equipped with the Hilbert
norm $H^0$ and 
\be
A\leq T_1
\ee
as operators on the space $H^0(X, L+K_X)$ equipped with the Hilbert
norm $H^1$.

We are now going to apply these estimates to multiples $kL$ of the
bundle $L$, but in order to accomodate also $L^2$-metrics of the form
$$
\int_X |u|^2 e^{-k\phi} \omega^\phi_n
$$
we need to generalize the set up first. Let therefore $F$ be an
arbitrary line bundle over $X$ and consider line bundles of the form
$$
K_X+F+kL.
$$
The main examples will be $F=0$ and $F=-K_X$, and the reader may find
it convenient to focus on the case $F=0$ first, in which case the
argument below is easier, at least notationally. Put now
$$
V_k=H^0(X, kL+F+K_X).
$$

Fix two metrics $\phi^0$ and $\phi^1$ in $\H_L$.
Let $\chi$ be some fixed metric on $L$ considered as a bundle over
$X\times\Omega$, which has positive curvature bounded from below by a
positive constant ( times say $\omega^{\phi^0}+ idt\wedge d\bar t$),
and which equals $\phi^0$ for $\Re t =0$ and equals $\phi^1$ for $\Re
t=1$. Such a metric $\chi$ can be found on the form
$$
t\phi^1 +(1-t)\phi^0 +\kappa(\Re t)
$$
where $\kappa$ is a sufficiently convex function on the interval
$(0,1)$ which equals 0 at the endpoints. 

Let also $\psi$ be an arbitrary metric on $F$, not necessarily with
positive curvature, but smooth up to the boundary. Choose a fixed
positive constant $a$, sufficiently large so that
$$
ai\ddbar\chi +i\ddbar\psi\geq 0.
$$

We next consider
the vector spaces 
$$
H^0(X, K_X+F+kL)
$$
with the induced $L^2$-metrics
$$
\|u\|^2_{k,t}:=\int_X |u|^2 e^{-(k-a)\phi -a\chi-\psi}.
$$
Notice that the metric on the line bundle $F+kL$ that we use here,
$(k-a)\phi +a\chi+\psi$ has been chosen so that it has nonnegative
curvature, meaning that we can apply the results from 3.1, 3.2 and 3.3.
We denote the Toepliz operators arising from differentiation of the
norms at $t=0$ and $t=1$ by $T_{0,k}$ and $T_{1,k}$ now in order to
keep track on how they depend on $k$. By immediate calculation
\be
\langle T_{k,t} u, u\rangle_{ k, t}=
-\int_X [(k-a)\p+ a\dot{\chi}+\dot{\psi}] |u|^2 e^{-(k-a)\phi -a\chi-\psi}
\ee
for $t=0, 1$. 

\bigskip

Let now $H^t_k$ be the finite dimensional geodesic in the space of
hermitian norms on $H^0(X, K_X+F+kL)$ that connects $\|\cdot\|_{k,t}$
for $t=0$ and $t=1$. Let 
$$
A_k= (H^t_k)^{-1}\frac{d}{dt}H^t_k
$$
be the tangent vector of the finite dimensional geodesic. By 3.2 and
3.3 we have the inequalities
\be
T_{0,k}\leq A_k
\ee
with respect to the hermitian scalar product $H^0_k$ and
\be
T_{1,k}\geq A_k
\ee
with respect to the hermitian scalar product $H^1_k$.  Let $\lambda_j(k)$ be the
eigenvalues of $A_k$ arranged in increasing order, and let
$\tau_j^t(k)$ be the eigenvalues of the two Toepliz operators, also
arranged in increasing order. We then get immediately from 3.5 and 3.6 
that
\be
\tau^0_j(k)\leq\lambda_j(k)\leq\tau^1_j(k).
\ee

\bigskip

The final step in the argument is the following theorem on the
asymptotics of Toepliz operators; it is a variant of a theorem of
Boutet de Monvel, \cite{B-G}. Since the theorem is essentially known,
we defer its proof to an appendix.

\begin{thm} Let $L$ and $F$ be line bundles over $X$ with smooth
  metrics $\phi$ and $\psi$ respectively. Assume that $\phi$ has
  strictly positive curvature. Let $\xi$ and $\xi_k$ be  continuous
  real valued functions on $X$ with $\xi_k$ tending uniformly to
  0. Define  Toepliz operators with 
  symbols $\xi + \xi_k$ on the
  spaces
$$
H^0(X, K_X+kL+F)
$$
by
$$
\langle T_k u,u\rangle_{k\phi+\psi}= \int (\xi + \xi_k)|u|^2 e^{-k\phi-\psi}.
$$
 Let $\mu_k$ be the normalized spectral measure of $T_k$.

Then the sequence $\mu_k$ converges weakly to the measure 
$$
\mu= \xi_*(\omega^\phi_n/Vol),
$$
the direct image of the normalized volume element on $X$ defined by
$\omega^\phi$ under the map $\xi$.
\end{thm}

We apply this theorem  to the Toepliz operator $k^{-1}T_{k,t}$ for
$t=0,1$. Its symbol is $-\p$ plus a term that goes uniformly to
zero.  In our operators  $k^{-1}T_{k,t}$ the metric on $F$ can be
taken to be $\psi +a(\chi-\phi)$ if we take the metric on $L$ to be
$\phi$. Theorem 3.2 therefore shows that the spectral measures
$d\mu_{k,t}$ of 
 $k^{-1}T_{k,t}$ converge to
$$
d\mu_t= (-\p)_*(dV_t).
$$ 

\bigskip

By the previous section these two measures are the same (for $t=0$ and
$t=1$), namely the
measure $d\mu$ that we associated to the geodesic in $\H_L$. The
inequality 3.7 for the eigenvalues shows that
$$
\int_\R f d\mu_{k,0}\leq \int_\R f dv_k\leq \int_\R f d\mu_{k,1}
$$
if $f$ is continuous and increasing (recall that $\nu_k$ is the
spectral measure of $A_k$). It follows that
$$
\lim \int_\R f d\nu_k = \int_\R f d\mu
$$
for $f$ continuous and increasing. Since any $C^1$-function can be written as a
difference of two increasing functions, the previous limit must hold
for any $C^1$-function too. But this implies weak convergence of the
measures since all the measures involved are probability measures
supported on a fixed compact interval. This finishes the proof of our
main result:
\begin{thm}
 Let $\phi^0$ and $\phi^1$ be two points in $\H_L$ and let $\psi$ be
 an arbitrary smooth metric on the line bundle $F$. Let 
$$
V_k=H^0(X,K_X+F+ kL)
$$
and let $\H_k$ be the space of hermitian norms on $V_k$.
Let $H^t_k$ be the elements in $\H_k$  defined by
$$
\|u\|^2=\int_X |u|^2 e^{-k\phi^t-\psi}
$$
for $t=0,1$. Let for $t$ between
0 and 1 $H^t_k$ be the geodesic in $\H_k$ connecting these two norms
and let $\nu_k$ be their normalized spectral measures as defined
above. Then 
$$
\nu_k \longrightarrow \mu,
$$
 in the weak* topology, where
$\mu=\mu_t$ is defined in 1.1.
\end{thm}

\bigskip

\noindent The basic observation in the proof is that the inequality
between finite dimensional geodesics and $L^2$-norms coming from
Monge-Ampere geodeics in Proposition 3.1 also gives inequality for the
first derivatives, since we have equality at the endpoint. The next
proposition (cf the sup norm estimate for $\p$ from
\cite{Phong-Sturm}) is another instance of this.
\begin{prop} With the same notation as in the previous theorem, and
$$
A_k =(H^t_k)^{-1}\dot{H}^t_k,
$$
let $\Lambda_{(k)}$ and $\lambda_{(k)}$ be the largest and smallest
eigenvalues of $k^{-1}A_k$. Then, for all $k$,
$$
 \inf -\p\leq\lambda_{(k)}\leq \Lambda_{(k)}\leq \sup -\p.
$$
\end{prop}

\begin{proof} This follows immediately from 3.7, since the
  corresponding inequality for the eigenvalues of the Toepliz
  operators is immediate.
\end{proof}

\section{Approximation of geodesics.}
Again we consider the spaces 
$$
V_k= H^0(X, K_X+F+kL)
$$
equipped with metrics
$$
\|u\|^2_{k\phi+\psi}:= \int_X|u|^2 e^{-k\phi-\psi}
$$
Let
$$
B_{k\phi+\psi}=\sum |s_j|^2,
$$
where $s_j$ is an orthonormal basis for $V_k$ . Since  pointwise
$$
|u|^2/B_{k\phi+\psi}
$$
is a function if $u$ is a section of $K_X+F+kL$,
$$
\log B_{k\phi+\psi}
$$
can be interpreted as a metric on $K_X+F+kL$. In the proof below we
will have use for the following lemma ( we formulate it for $F=0$
and $k=1$), which is a variant on a  well known theme. The basic
underlying idea, to estimate Bergman kernels using the
Ohsawa-Takegoshi theorem is due to Demailly, see e g \cite{Demailly}.
\begin{lma} Let $\omega^0$ be a fixed K\"ahler form on $X$. Let $\phi$
  be a metric (not necessarily smooth) on the line bundle $L$ satisfying
$$
i\ddbar\phi\geq c_0\omega^0.
$$
Let $H_\phi$ be the norm
$$
\int_X |u|^2 e^{-\phi}
$$
for $u$ in $H^0(X, L+K_X)$, and let $B_\phi$ be its Bergman
kernel. Then 
$$
B_\phi\geq \delta_0 e^\phi \omega^0_n
$$
with $\delta_0$ a universal constant, if $c_0$ is sufficiently large
depending on $X$ and $\omega^0$ (only).
\end{lma}
\begin{proof} By the extremal characterization of Bergman kernels it
  suffices to find a section $u$ of $K_X+L$ with
$$
|u(x)|^2 e^{-\phi(x)}\geq \delta_0\omega^0_n \int_X |u|^2 e^{-\phi}
$$
Choose a coordinate neighbourhood $U$ centered at $x$ which is
biholomorphic to the unit ball of $\C^n$. By the Ohsawa-Takegoshi
extension theorem we can find a section satisfying the required
estimate over $U$. Let $\eta$ be a cut-off function, equal
to 1 in the ball of radius 1/2 and with compact support in the unit
ball. We then solve, using H\"ormander's $L^2$-estimates
$$
\dbar v=\dbar\eta\wedge u=:g
$$
with  
$$
\int_X|v|^2 e^{-\phi-2n\eta\log|z|} \leq (C/c_0)\int_X|g|^2
e^{-\phi-2n\eta\log|z|}
$$ 
( $z$ is the local coordinate). This can be done since
$$
i\ddbar\phi-2n\eta\log|z|\geq c_0 \omega^0/2
$$
if $c_0$ is large enough. Then $v(x)=0$ since the integral in the left
hand side is finite. Then
$$
u -v
$$
is a global holomorphic section of $K_X+L$ satisfying the required
estimate. 
\end{proof}

\bigskip

\noindent Let $\phi^0$ and $\phi^1$ be two points in $\H_L$, and let
$\psi$ be any smooth metric on $F$. We abbreviate by $H^t_k$ the norms
$\|\cdot\|_{k\phi^t+\psi} $ for $t$ equal to 0 or 1 , and let for $t$
between 0 and 1 $H^t_k$ be the geodesic in $\H_k$, the space of
hermitian norms on $V_k$, joining these two endpoints. 
\begin{thm} Let $\phi^t$ be two points in $\H_L$ for $t$ equal to 0
  and 1, and let for $t$ between 0 and 1 $\phi^t$ be the geodesic in
  $\H_L$ joining them. Let $B_{t, k}$ be the Bergman kernels for the
  norms $H^t_k$. Let $\tau$ be an arbitrary smooth metric on $K_X+F$
  over $\Omega\times X$. Then
$$
\sup_X |k^{-1}\log B_{t, k}-k^{-1}\tau -\phi^t|\leq C k^{-1}\log k
$$
for $0\leq t\leq 1$
\end{thm}

\bigskip

\noindent If $F=0$  this is exactly Theorem 6.1 in \cite{Berndtsson};
if $F=-K_X$ (so we can take $\tau=0$)  it is Theorem 1.2 from the introduction. 

\begin{proof} As just explained $\log B_{t, k} $
is a metric  on $K_X+F+kL$ and moreover
$$
i\ddbar \log B_{t, k} \geq 0.
$$
The last fact follows since $H^t_k$ are geodesics.  Perhaps the easiest
way to see it is to use the explicit description
$$
B_{t,k}=\sum |e^{-t\lambda_j}||s_j|^2
$$
which is immediate from the explicit formula for geodesics in section
2. Thus
$$
k^{-1}(\log B_{t, k}- \tau)
$$
is a metric on $L$. We shall now us the metric $\chi$ on $L$ that we
introduced in the previous section;  it has strictly positive
curvature over $\Omega\times X$ and coincides with $\phi^0$ and
$\phi^1$ respectively when $(\Re) t$ is 0 or 1. Take $a$ to be positive
and consider
$$
(k-a)k^{-1}(\log B_{t, k}- \tau)+ a\chi;
$$
it is a smooth metric on $kL$ and it has positive curvature if $a$ is
sufficiently large. By standard Bergman kernel asymptotics
it differs from $\phi^0$ and $\phi^1$ at most by $C\log k$ when
$(\Re)t$ equals 0 or 1. Hence
$$
(k-a)k^{-1}(\log B_{t, k}- \tau)+ a\chi \leq k\phi^t +C\log k
$$
since the geodesic $\phi^t$ is the supremum of all positively curved
metrics lying below $\phi^0$ and $\phi^1$ on the boundary (cf
\cite{Chen}). Dividing by $(k-a)$ we see that
$$
k^{-1}\log B_{t, k}-k^{-1}\tau -\phi^t \leq Ck^{-1}\log k
$$
since $\chi$, $\tau$ and $\phi^t$ are all uniformly bounded. 
The crux of the proof is the opposite estimate.

\bigskip

\noindent To estimate $B_{t k}$  from below we first compare it to the
Bergman kernel 
$$
B_{\phi^t, k},
$$
which is defined using the hermitian norms
$$
\|u\|^2_{*}=\int_X|u|^2 e^{-(k-a)\phi^t-a\chi-\psi}.
$$
Again, the metric $ (k-a)\phi^t+ a\chi +\psi$ that we use here has
positive curvature if $a$ is sufficiently large. These norms coincide
with $H^t_k$ on the boundary and by Proposition 3.1 they are bigger
than $H^t_k$ in the interior. This implies (by the extremal
characterization of Bergman kernels) that the respective Bergman
kernels satisfy the opposite inequality, so we get
$$
\log B_{t, k}\geq \log B_{\phi^t, k}.
$$
To complete the proof it therefore suffices to show that
$$
B_{\phi^t, k}\geq C e^{k\phi^t+\tau},
$$
or equivalently
$$
B_{\phi^t, k}\geq C e^{(k-a)\phi^t +a\chi+\tau}
$$
But this follows from Lemma 4.1 since we can take $a$ arbitrarily
large so that
$$
i\ddbar(k-a)\phi^t +a\chi+\tau
$$
meets the curvature assumptions of that lemma.
\end{proof}

\section{ Appendix: Background on Toepliz operators.}
We consider Toepliz operators $T_{k,\xi}$ on the spaces
$$
V_k=H^0(X, K_X+F+kL)
$$
with symbol $\xi$ in $C(X)$. $T_{k,\xi}$ is defined by
$$
\langle T_{k,\xi}u,u\rangle_{k\phi+\psi} = \int_X \xi |u|^2
e^{-k\phi-\psi},
$$
where the inner product is
$$
\langle v,u\rangle_{k\phi+\psi} = \int_X v \bar u e^{-k\phi-\psi}.
$$
In other words
$$
T_{k,\xi} u= P_k(\xi u)
$$
where $P_k$ is the Bergman projection.

\medskip

Recall that if $T$ is any hermitian endomorphism on an $N$-dimensional inner
product space, and if we order its eigenvalues
$$
\lambda_\leq \lambda_2\leq ... \lambda_N,
$$
then
$$
\lambda_j = \inf_{Vj\subset V, dim V_j =j} \|T|_{V_j}\|.
$$
From this it follows that if we perturb the operator $T$ to $T+S$
where $\|S\|\leq \epsilon$, then the eigenvalues shift at most by
$\epsilon$. This means that if we consider the spectral measure of
$$
T_{k,\xi+\xi_k}
$$
where $\xi_k$ goes uniformly to 0, the limit of the spectral measures
is the same as the limit of the spectral measures of
$$
T_{k,\xi}.
$$
In other words, in the proof of Theorem 3.2 we may assume that
$\xi_k=0$. By the same token, we may assume that $\xi$ is smooth,
since continuous functions can be approximated by smooth functions. 
The most important part of the proof of Theorem 3.2 is the next lemma.

\begin{lma} Let $d_k= dim(V_k)$. Then 
$$
\lim \frac{1}{d_k} tr T_{k,\xi}=\int_X\xi \omega^\phi_n/Vol.
$$
\end{lma}

\begin{proof} Let $B_{k\phi+\psi}$ be the Bergman kernel. Then
$$
 \frac{1}{d_k} tr T_{k,\xi}= \frac{1}{d_k}\int_X \xi B_{k\phi+\psi}
 e^{-k\phi-\psi} .
$$
But, by the formula for (first order) Bergman asymptotics
$$
 B_{k\phi+\psi} e^{-k\phi-\psi}/d_k
$$
tends to $\omega^\phi_n/Vol$, so the lemma follows.
\end{proof}

\begin{lma} Let $\xi$ and $\eta$ be smooth functions on $X$. Then
$$
\| T_{k,\xi} T_{k, \eta} - T_{k, \,\xi\eta}\|^2 \leq C k^{-1}.
$$
\end{lma}
\begin{proof} Note that if $u$ is in $V_k$ then
$$
T_{k, \xi} u - \xi u =: v_k
$$
is the $L^2$-minimal solution to the $\dbar$-equation
$$
\dbar v_k =\dbar\xi\wedge u
$$
(this is where we want $\xi$ smooth). By H\"ormander $L^2$-estimates
$$
\|T_{k, \xi} u - \xi u\|^2_{k\phi+\psi} \leq \|\dbar\xi\wedge
u\|^2_{k\phi+\psi}  \leq C k^{-1}\|u\|^2_{k\phi+\psi} 
$$
(the last inequality is because the pointwise norm
$\|\dbar\xi\|^2_{\theta}\leq C/k$ when we 
measure with respect to the K\"ahler metric $\theta= i\ddbar(k\phi+\psi)$). 
Therefore, if $u$ is of norm at most 1,
$$
\| T_{k,\xi} T_{k, \eta}u -\xi T_{k,\eta}u\|^2 \leq C k^{-1},
$$
$$
\|\xi T_{k,\eta}u -\xi\eta u\|^2 \leq C k^{-1}
$$
and
$$
\| T_{k,\, \xi\eta}u  -\xi \eta u\|^2 \leq C k^{-1}
$$
and the lemma follows.
\end{proof}
\medskip
Let $\mu_k$ be the normalized spectral measures of $T_{k, \xi}$. In
order to study their weak limits, it is enough to look at their
moments
$$
\int_\R x^p d\mu_k(x) =\frac{1}{d_k} tr T^p_{k, \xi}.
$$
By Lemma 7.2 and induction
$$
\|T_{k, \xi}^p- T_{k, \xi^p}\|^2 \leq C k^{-1}.
$$
Hence
$$
\frac{1}{d_k} tr T^p_{k, \xi}=\frac{1}{d_k} tr T_{k, \xi^p} +
O(k^{-1})
$$
and
$$
\lim \frac{1}{d_k} tr T_{k, \xi^p}= \int_X \xi^p \omega^\phi_n/Vol
$$
by Lemma 7.1. Thus,
$$
\lim \int_\R x^p d\mu_k(x)= \frac{1}{d_k} tr T^p_{k, \xi} =   \int_X
\xi^p\omega^\phi_n/Vol 
$$
for any power $x^p$. Taking linear combinations we get the same thing
for any polynomial  , and therefore for any continuous function. This
completes the proof of Theorem 3.2.

\def\listing#1#2#3{{\sc #1}:\ {\it #2}, \ #3.}

\end{document}